\documentclass[11pt]{amsart}
\usepackage{latexsym}
\usepackage{palatino}
\usepackage{amsmath, amssymb}
\usepackage[dvipsnames,usenames]{color}
\usepackage[all]{xy}
\newtheorem{dfs}{Definition}[section]
\newtheorem{lms}[dfs]{Lemma}
\newtheorem{thms}[dfs]{Theorem}
\newtheorem{props}[dfs]{Proposition}

\newtheorem{cors}[dfs]{Corollary}
\newtheorem{rems}[dfs]{Remark}

\newtheorem*{thm*}{Theorem}

\title[$\mathrm{K}$-theoretic rigidity and slow dimension growth]
{$\mathrm{K}$-theoretic rigidity and slow dimension growth}

\address{Department of Mathematics \\ Purdue University \\ 150 N. University St. \\ West Lafayette IN \\ USA \\ 47907 }

\email{atoms@purdue.edu}

\begin{document}

\begin{abstract}
Let $A$ be an approximately subhomogeneous (ASH) C$^*$-algebra with slow dimension growth.  We prove that if $A$ is unital and simple, then the Cuntz semigroup of $A$ agrees with that of its tensor product with the Jiang-Su algebra $\mathcal{Z}$.  In tandem with a result of W. Winter, this yields the equivalence of $\mathcal{Z}$-stability and slow dimension growth for unital simple ASH algebras.  This equivalence has several consequences, including the following classification theorem:  unital ASH algebras which are simple, have slow dimension growth, and in which projections separate traces are determined up to isomorphism by their graded ordered $\mathrm{K}$-theory, and none of the latter three conditions can be relaxed in general.  
\end{abstract}

\maketitle

 \section{Introduction and statement of main results}
 A C$^*$-algebra is {\it subhomogeneous} if there is a uniform finite bound on the dimensions of its irreducible representations, and {\it approximately subhomogeneous (ASH)} if it is the limit of a direct system of subhomogeneous C$^*$-algebras.  ASH algebras form a broad class with many naturally occurring examples:
 \begin{itemize}
 \item[$\bullet$]  AF algebras, which include the simple stably finite C$^*$-algebras of graphs (\cite{R}).
 \item[$\bullet$]	 C$^*$-algebras of minimal dynamical systems on finite-dimensional spaces which are either smooth or uniquely ergodic (\cite{LP}, \cite{TW2}, \cite{TW3}).
 \item[$\bullet$] Higher-dimensional noncommutative tori (\cite{P3}).
 \item[$\bullet$] The homoclinic and heteroclinic C$^*$-algebras of 1-solenoids (\cite{T}).
 \end{itemize}
 In fact, there are no simple separable nuclear stably finite C$^*$-algebras which are known not to be ASH.
 
 This article characterizes the unital separable ASH algebras that are determined up to isomorphism by their graded ordered $\mathrm{K}$-groups.  Of necessity, one considers only algebras in which projections separate traces, as the tracial state space of the algebra will otherwise be part of any complete invariant.  Elliott conjectured c. 1990 that modulo this necessary assumption, all unital simple separable ASH algebras would be determined by their $\mathrm{K}$-groups.  We now know that this conjecture, while true in considerable generality, is too much to hope for.  The author showed in \cite{To1} and \cite{To2} that an additional condition---{\it slow dimension growth}---is required in general, a condition present in each of the examples listed above.  Finally, one needs simplicity in order to avoid phenomena detectable only using $\mathrm{K}$-theory with $(\mathrm{mod} \ p)$-coefficients (see \cite{DG} and \cite{E}).  We conclude here that these three necessary conditions are also sufficient.  
 
Our route passes through the Cuntz semigroup, an ordered Abelian semigroup consisting of equivalence classes of countably generated Hilbert modules over a C$^*$-algebra.  For a C$^*$-algebra $A$, this semigroup is denoted by $W(A)$.  Winter has proved the following remarkable theorem.
 \begin{thms}[Winter, \cite{Wi3}]\label{locfin}
 Let $A$ be a unital simple separable C$^*$-algebra with locally finite nuclear dimension.  If $W(A) \cong W(A \otimes \mathcal{Z})$, then $A \cong A \otimes \mathcal{Z}$.
 \end{thms}
 \noindent
 Here $\mathcal{Z}$ denotes the Jiang-Su algebra (\cite{JS}).  Tensorial absorption of $\mathcal{Z}$---known as {\it $\mathcal{Z}$-stability}---is crucial for lifting $\mathrm{K}$-theory isomorphisms to C$^*$-algebra isomorphisms (see \cite{ET} for a discussion of this connection).  We will not define locally finite nuclear dimension here;  it is enough for us that separable ASH algebras have it (\cite{NW}, \cite{Wi1}, \cite{Wi2}).  We access Theorem \ref{locfin} with our main result.
 \begin{thms}\label{waz}
 Let $A$ be a unital simple separable ASH algebra with slow dimension growth.  It follows that 
 $W(A) \cong W(A \otimes \mathcal{Z})$.
 \end{thms}
\noindent
The property of slow dimension growth appeared first in the early 1990s in connection with so-called AH algebras (a subclass of ASH algebras which model higher-dimensional noncommutative tori, for instance).  It was first seen as a natural condition ensuring weak unperforation of the ordered $\mathrm{K}_0$-group and the density of invertible elements in simple AH algebras (\cite{BDR}, \cite{DNNP}), and later proved to be critical for obtaining classification-by-$\mathrm{K}$-theory results (\cite{D}, \cite{EG}, \cite{G1}, \cite{G2}).  Philosophically, it excludes the possibility of unstable homotopy phenomena.  As a corollary of Theorems \ref{locfin} and \ref{waz} we obtain the following result.
 \begin{cors}\label{zstabsdg}
 Let $A$ be a unital simple separable ASH algebra.  It follows that $A$ has slow dimension growth if and only if $A \cong A \otimes \mathcal{Z}$.
 \end{cors}
 \noindent
$\mathcal{Z}$-stability is a necessary and in considerable generality sufficient condition for the classification of nuclear simple separable C$^*$-algebras via $\mathrm{K}$-theory and traces, while slow dimension growth has been conjectured to play a similar role for the subclass of simple ASH algebras.  Corollary \ref{zstabsdg} confirms this conjecture after a fashion:  their roles are at least identical.  (Winter showed that $A$ as in Corollary \ref{zstabsdg} satisfies $A \cong A \otimes \mathcal{Z}$ whenever $A$ satisfies the formally stronger condition of bounded dimension growth (\cite{Wi4});  the question of whether the reverse implication holds is open.)  Corollary \ref{zstabsdg} has several further consequences for a unital simple separable ASH algebra $A$ with slow dimension growth;  we give a brief run-down here, with references to fuller details.
 \begin{itemize}
 \item[$\bullet$] $A$ has stable rank one, answering an open question of Phillips from \cite{P2}.  In fact, all of the conclusions of \cite[Theorem 0.1]{P2} hold for $A$;  in particular, the extra conditions of items (4) and (5) in that Theorem are not necessary.
 \item[$\bullet$] The Blackadar-Handelman conjectures hold for $A$, i.e., the lower semicontinuous dimension functions on $A$ are weakly dense in the space of all dimension functions, and the latter space is a Choquet simplex.  (See Section 6 of \cite{BPT}.)
 \item[$\bullet$] The countably generated Hilbert modules over $A$ are classified up to isomorphism by the $\mathrm{K}_0$-group and tracial data in a manner analogous to the classification of W$^*$-modules over a $\mathrm{II}_1$ factor.  (See \cite[Theorem 3.3]{BT}.)
 \item[$\bullet$] The Cuntz semigroup of $A$ is recovered functorially from its $\mathrm{K}_0$-group and tracial state space.  (See \cite[Theorem 2.5]{BT} and the comment thereafter.)
 \end{itemize}

\pagebreak

Finally, we have the classification result.  
\begin{cors}\label{class}
 Let $\mathcal{C}$ denote the class of all unital simple separable ASH algebras with slow dimension growth in which projections separate traces.  If $A,B \in \mathcal{C}$ and 
 \[
 \phi:\mathrm{K}_*(A) \to \mathrm{K}_*(B)
 \]
 is a graded order isomorphism, then there is a $*$-isomorphism $\Phi:A \to B$ which induces $\phi$.
 \end{cors}
 \noindent
As mentioned above, the conditions of simplicity, slow dimension growth, and the separation of traces by projections are necessary in general\footnote{It is conjectured that the separation of traces by projections can be dropped from the hypotheses of Corollary \ref{class} if one augments the invariant $\mathrm{K}_*$ by the simplex of tracial states.  Corollary \ref{zstabsdg} and the results of \cite{Wi5} make some progress on this conjecture by showing that it need only be addressed for algebras which absorb a fixed UHF algebra of infinite type tensorially.  This problem should in turn be accessible to tracial approximation techniques in the spirit of Lin (see \cite{L}, for instance).}.  This result is satisfying not only for its completeness, but also because it represents the first time that the structure of the Cuntz semigroup has played a critical role in a positive classification theorem for simple C$^*$-algebras.  Its proof combines Corollary \ref{zstabsdg} with results of Lin, Niu, and Winter (\cite{LN}, \cite{Wi2}, \cite{Wi5}).

The sequel is given over to the proof of Theorem \ref{waz}.  By appealing to some known results concerning the structure of the Cuntz semigroup, the crux can be reduced to the following natural question: \\

\noindent
\hspace{8mm}Given a unital simple ASH algebra $A$ with slow dimension growth, what are the \\ \hspace*{7mm} possible ranks of positive operators in $A \otimes \mathcal{K}$?\\

\noindent
Here by the rank of a positive operator $a \in A \otimes \mathcal{K}$ we mean the function on the tracial state space of $A$ given by
\[
\tau \mapsto d_\tau(a) = \lim_{n \to \infty} \tau(a^{1/n}).
\]
We prove that every strictly positive lower semicontinuous affine function occurs in this manner by giving an approximate answer to the same question for {\it recursive subhomogeneous C$^*$-algebras}, the building blocks of ASH algebras.  This, in turn, requires proving that the homotopy groups of certain rank-constrained sets of positive operators in $\mathrm{C}(X) \otimes \mathcal{K}$ vanish in low dimensions (Section \ref{rankhom}).   The proofs of Theorem \ref{waz} and Corollaries \ref{zstabsdg} and \ref{class} are contained in Section \ref{main}.

  \vspace{3mm}
  \noindent
  {\it Acknowledgements.}   We thank Wilhelm Winter for several valuable conversations on the topic of this article had during a visit to the University of Nottingham in June of 2009. 
  
   \section{Rank-constrained homotopies}\label{rankhom} 
The main result of this Section is Proposition \ref{envelopeextend}.  It allows one to extend a positive element in a matrix algebra over a closed subset $Y$ of a compact metric space $X$ to all of $X$ subject to a pair of rank bounds given by a lower and an upper semicontinuous $\mathbb{Z}$-valued function on $X$. 

\begin{lms}\label{ranketa}
Let $X$ be a compact metric space, and let $a \in \mathrm{M}_n(\mathrm{C}(X))$ be positive.  Let $g:X \to \mathbb{Z}^+$ be upper semicontinuous, and suppose that 
\[
\mathrm{rank}(a(x)) \geq g(x), \ \forall x \in X.
\]
It follows that for some $\eta>0$, for each $x \in X$, the spectral projection $\chi_{(\eta,\infty]}(a(x))$ has rank at least $g(x)$. 
\end{lms}

\begin{proof}
For each $x \in X$, let $\eta_x \geq 0$ be half of the smallest nonzero eigenvalue of $a(x)$, if it exists, and zero otherwise.  
The map $x \mapsto \mathrm{rank}(a(x))$ is lower semicontinuous, 
so there is an open neighbourhood $V_x$ of $x$ with the property that 
\[
\mathrm{rank}[\chi_{(\eta_x,\infty]}(a(y))] \geq g(x), \ \forall y \in V_x.
\]
Since $g$ is upper semicontinuous and $\mathbb{Z}$-valued, there is an open neighbourhood $W_x$ of $x$ such that $g(y) \leq g(x)$ for each $y \in W_x$.  Set $U_x = V_x \cap W_x$.
Now
\begin{equation}\label{etax}
\mathrm{rank}[\chi_{(\eta_x,\infty]}(a(y))] \geq g(x) \geq g(y), \ \forall y \in V_x.
\end{equation}
Since $\{V_x \ | \ x \in X\}$ is an open cover of $X$, it admits a finite subcover $V_{x_1} \cup \cdots \cup V_{x_n}$.  Let $\eta$ be the minimum of the
nonzero $\eta_{x_i}$s.  Now $\mathrm{rank}[\chi_{(\eta,\infty]}(a(x))] \geq g(x)$ on each $V_{x_i}$ such that $\eta_{x_i} > 0$ by (\ref{etax}), and the 
same inequality holds on the remaining $V_{x_i}$ since $g$ is identically zero on these sets.
\end{proof}

From here on we use $\mathrm{dim}(X)$ to denote the covering dimension of a compact Hausdorff space $X$.  We also recall that a projection-valued map $\phi:X \to \mathrm{M}_n$ is said to be lower semicontinuous (resp. upper semicontinuous) if the map $x \mapsto \langle \phi(x) \xi, \xi \rangle$ is lower semicontinuous (resp. upper semicontinuous) for every $\xi \in \mathbb{C}^n$.

\begin{lms}\label{selecttriv}
Let $X$ be a compact Hausdorff space for which $\mathrm{dim}(X)<\infty$, and let $a \in \mathrm{M}_n(\mathrm{C}(X))$ be positive.  Suppose that 
\[
\mathrm{rank}(a(x)) \geq k, \ \forall x \in X.
\]
It follows that there is a homotopy $h:[0,1] \to \mathrm{M}_n(C(X))_+$ with the following properties:
\begin{enumerate}
\item[(i)] $h(0) = a$;
\item[(ii)] $\mathrm{rank}(h(t)(x)) = \mathrm{rank}(h(0)(x)) = \mathrm{rank}(a(x)), \ \forall x \in X$;
\item[(iii)] there is a trivial projection $p \in \mathrm{M}_n(\mathrm{C}(X))$ of rank at least $k - \mathrm{dim}(X)$ which is a direct summand of $h(1)$.
\end{enumerate}
\end{lms}

\begin{proof}
Let $a$ be given.  We may assume that $\|a\| \leq 1$.  Use Lemma \ref{ranketa} to find $\eta>0$ such that the rank of $\chi_{(\eta,\infty]}(a(x))$ is at least $k$ for 
each $x \in X$.  For each $s \in (0,1]$ define
a continuous map $f_s:[0,1] \to [0,1]$ by insisting that $f_s$ is identically equal to one on $[s,1]$, that $f_s(0)=0$, and that $f_s$ is linear elsewhere.  Let $s = 1- t(1-\eta/2)$, and define $h(t)(x) = f_{s}(a(x))$.  This is clearly a homotopy.  When $t=0$, $s=1$, so $f_1(a(x)) = a(x)$ and so $h(0)=a$ as required by (i).  Since the support of $f_s$ is 
$(0,1]$, we have that $\mathrm{rank}(f_s(a(x)))=\mathrm{rank}(a(x))$ for each $x \in X$ and $s \in (0,1]$, establishing (ii).  

To prove (iii), first note that $x \mapsto \chi_{(\eta,\infty]}(a(x))$ is a lower semicontinuous 
projection-valued map having rank at least $k$ at each $x \in X$.  Since
$\chi_{(\eta,\infty]}(a(x)) \leq f_{\eta/2}(a(x))$ for each $x$, we see by functional calculus that 
$\chi_{(\eta,\infty]}(a(x))$ is a direct summand of $h(1)(x)$ for each $x \in X$. 
It follows from Proposition 3.2 of \cite{DNNP} that there is a continuous projection-valued map $q:X \to \mathrm{M}_n$ which 
is pointwise a direct summand of $\chi_{(\eta,\infty]}(a)$ and satisfies $\mathrm{rank}(q) \geq k - \lfloor (\mathrm{dim}(X)-1)/2 \rfloor$.  
It is well known that such a $q$ admits a direct summand $p$ which corresponds to a trivial vector bundle 
and satisfies $\mathrm{rank}(p) \geq \mathrm{rank}(q) - \lfloor (\mathrm{dim}(X)-1)/2 \rfloor$.  Note that $p(x)$ is a direct summand of $q(x)$, that $q(x)$ is a direct summand of $\chi_{(\eta,\infty]}(a(x))$, and, as noted above, that $\chi_{(\eta,\infty]}(a(x))$ is a direct summand $h(1)(x)$;  it follows that $p$ is a direct summand of $h(1)$ as required.  The preceding rank inequality between $p$ and $q$ entails that 
$\mathrm{rank}(p) \geq k-\mathrm{dim}(X)$.
\end{proof}

\begin{dfs}[ Definition 3.4 (iii), \cite{To3}]\label{wellsupported}
Let $X$ be a compact Hausdorff space and let $a \in \mathrm{M}_n(\mathrm{C}(X))$ be positive.  Let $n_1 <  n_2 < \ldots  < n_k$ be the rank values taken by $a$ on $X$, and set 
\[
E_i = \{x \in X \ | \ \mathrm{rank}(a(x)) = n_i \}.
\]
We say that $a$ is well supported if there are constant rank projections $p_i \in \mathrm{M}_n(\mathrm{C}(\overline{E_i}))$ with the following properties:
\begin{itemize}
\item[$\bullet$] $p_i(x) \leq p_j(x)$ whenever $i \leq j$ and $ x \in \overline{E_i} \cap \overline{E_j}$;
\item[$\bullet$] $p_i(x) = \lim_{n \to \infty} (a(x))^{1/n}$ for each $x \in E_i$.
\end{itemize}
\end{dfs}

In the next Lemma and elsewhere, we use ``$\precsim$'' to denote the Cuntz relation on the positive elements of a C$^*$-algebra (see \cite{ET}, for instance). 

\begin{lms}\label{rankincrease}
Let $X$ be a compact Hausdorff space, and let $a \in \mathrm{M}_n(\mathrm{C}(X))$ be positive.  Suppose that 
\[
l \leq \mathrm{rank}(a(x)) \leq k, \ \forall x \in X.
\]
for $l,k \in \mathbb{N}$ satisfying $k \leq n$, $l \leq \mathrm{dim}(X)$ and $4 \mathrm{dim}(X) \leq k-l$.
It follows that there is a homotopy $h:[0,1] \to \mathrm{M}_n(C(X))_+$ with the following properties:
\begin{enumerate}
\item[(i)] $h(0) = a$;
\item[(ii)] $l \leq \mathrm{rank}(h(t)(x)) \leq k, \ \forall x \in X, t \in [0,1]$;
\item[(iii)] $\mathrm{rank}(h(1)(x)) \geq l+\mathrm{dim}(X), \ \forall x \in X$.
\end{enumerate}
\end{lms}

\begin{proof}
Use Lemma \ref{ranketa} to find $\eta>0$ such that $\chi_{(\eta,\infty]}(a(x))$ has rank at least $l$ at each $x \in X$.  
By \cite[Theorem 2.3]{To4}, there is a well supported positive element $b$ of $\mathrm{M}_n(\mathrm{C}(X))$ such that $b \leq a$ and $\| 
b-a \| < \eta$.  Set $a_t = (1-t)a + t b$.  Since $a_t \leq a$, we have 
\[
\mathrm{rank}(a_t(x)) \leq \mathrm{rank}(a(x)) \leq k, \ \forall x \in X.
\]
On the other hand, we have $\|a_t -a \| < \eta$, whence $(a-\eta)_+ \precsim a_t$ for each $t$.  Now by our choice of $\eta$, we have
\[
l \leq \mathrm{rank}((a-\eta)_+(x)) \leq \mathrm{rank}(a_t(x)), \ \forall x \in X.
\]
We therefore have the bounds required by part (ii) of the conclusion of the Lemma for the homotopy $a_t$.  It follows that we may simply assume that the element $a$ of the Lemma is well supported from the outset.  Let $n_1 < n_2 < \cdots < n_k$, $E_1, E_2,\ldots,E_k$, and $p_1, p_2, \ldots, p_k$ be as in Definition \ref{wellsupported}.  To complete the proof of the Lemma, we treat two cases.

\vspace{2mm}
\noindent
{\bf Case I.}  Here we assume that $S = \{x \in X \ | \ \mathrm{rank}(a(x)) > l+2\mathrm{dim}(X) \}$ is empty.  The upper semicontinuous projection-valued map $\phi: X \to \mathrm{M}_n$ given by 
\[
\phi(x) = \bigvee_{i=1}^k p_i(x)
\]
therefore has rank less than or equal to $l+ 2\mathrm{dim}(X)$ everywhere, and it follows from \cite[Theorem 3.1]{BE} that there is a projection $p \in \mathrm{M}_n(\mathrm{C}(X))$ which is orthogonal to the image of $\phi$ and satisfies $\mathrm{rank}(p(x)) \geq \mathrm{dim}(X)$.  It is now easy to check that the homotopy $h(t) = a + tp$ satisfies (i)--(iii) in the conclusion of the Lemma.

\vspace{2mm}
\noindent
{\bf Case II.}  Assume that $S = \{x \in X \ | \ \mathrm{rank}(a(x)) > l+2\mathrm{dim}(X) \} \neq \emptyset$.  Let $r$ be the smallest index for which $n_r > l+ 2\mathrm{dim}(X)$, and set $Y = \cup_{i \geq r} \overline{E_i}$.  The lower semicontinuous projection-valued map $\psi:X \to \mathrm{M}_n$ given by
\[
\psi(x) = \left \{ \begin{array}{ll} \bigwedge_{\{i \ | \ i \geq r, x \in \overline{E_i} \}} p_i(x), & x \in Y \\ \mathbf{1}_{\mathrm{M}_n}, & x \in X \backslash Y \end{array} \right.
\]
therefore has rank strictly greater than $l+2\mathrm{dim}(X)$ everywhere, and it follows from \cite[Proposition 3.2]{DNNP} that there is a projection $q \in \mathrm{M}_n(\mathrm{C}(X))$ satisfying $\mathrm{rank}(q) = l+\mathrm{dim}(X)$ and $q(x) \leq p_i(x)$ whenever $x \in E_i$ and
$i \geq r$.  Set $h(t) = a + tq$.  

If $x \in E_i$ and $i \geq r$, then 
\[
h(t)(x) = a(x) + tq(x) \leq 2a(x) \precsim a(x),
\]
and so $\mathrm{rank}(h(t)(x)) \leq \mathrm{rank}(a(x)) \leq k$.  
If $i < r$, then 
\[
\mathrm{rank}(h(t)(x)) \leq \mathrm{rank}(a(x)) + \mathrm{rank}(q(x)) \leq 2l+ 3\mathrm{dim}(X) \leq l+4\mathrm{dim}(X)=k.
\]
If $t=1$, then 
\[
h(t)(x) = a(x) + q(x) \geq q(x),
\]
whence $\mathrm{rank}(h(t)(x)) \geq l+\mathrm{dim}(X)$.  This completes the proof.

\end{proof}

\begin{props}\label{homandrank}
Let $X$ be a compact Hausdorff space for which $\mathrm{dim}(X)<\infty$, and let $k,l,n \in \mathbb{N}$ satisfy $k \leq n$ and $4\mathrm{dim}(X) \leq k-l$.  
It follows that the set 
\[
S = \{ a \in \mathrm{M}_n(\mathrm{C}(X))_+ \ | \  l \leq \mathrm{rank}(a(x)) \leq k, \ \forall x \in X \}
\]
is path connected.
\end{props}

\begin{proof}
Let $a,b \in S$.
If $l \leq \mathrm{dim}(X)$, then by Lemma \ref{rankincrease} we may assume that 
\[
\mathrm{rank}(a(x)) \geq l+ \mathrm{dim}(X)
\]
for each $x \in X$.  If $l > \mathrm{dim}(X)$, then use Lemma \ref{selecttriv} to see that $a$ is homotopic inside $S$ to $a_1 = a_2 \oplus p$, where $p$ is a trivial projection of rank $l-\mathrm{dim}(X)$ and $a_1$ is positive.  Now use Lemma \ref{rankincrease} to find a homotopy 
\[
h(t) \in (1-p)(\mathrm{M}_n(\mathrm{C}(X))(1-p) \cong \mathrm{M}_{n-\mathrm{rank}(p)}(\mathrm{C}(X))
\]
between $a_2$ and $a_3:=h(1)$.  (Note that $1-p$ corresponds to a trivial vector bundle because it has the correct $\mathrm{K}_0$-class and satisfies $\mathrm{rank}(1-p) \geq \mathrm{dim}(X)/2$.)  It follows that $g(t) = h(t) \oplus p$ is a homotopy in $S$, and that $a_4 := g(1)$ satisfies
\[
l+ \mathrm{dim}(X) \leq \mathrm{rank}(a_4(x)) \leq k, \ \forall x \in X.
\]
Thus, we may assume that $a$ and $b$ satisfy
\[
l+\mathrm{dim}(X) \leq \mathrm{rank}(a(x)), \ \mathrm{rank}(b(x)) \leq k, \ \forall x \in X.
\]

Use Lemma \ref{selecttriv} again to see that $a$ is homotopic inside $S$ to $a_5 = a_6 \oplus q$, where $a_5$ is positive and $q$ is a trivial projection of rank $l$.  The upshot of these observations is that we may assume from the outset that
\[
a = \tilde{a} \oplus q \ \ \mathrm{and} \ \ b = \tilde{b} \oplus q^{'},
\]
where $q$ and $q^{'}$ are trivial projections of rank $l$.   From stable rank considerations there is a path $u(t)$ of unitaries in $\mathrm{M}_n$ such that $u(0)qu(0)^* = q$ and $u(1)qu(1)^* = q^{'}$.  We may therefore assume further that $q = q^{'}$.  Define a homotopy $t \mapsto a_t$ in 
\[
(1-q)(\mathrm{M}_n(\mathrm{C}(X))(1-q) \cong \mathrm{M}_{n-\mathrm{rank}(q)}(\mathrm{C}(X))
\]
by the following formula:
\[
a_t = \left\{ \begin{array}{ll} (1-2t)\tilde{a}, & t \in [0,1/2] \\ (2t-1)\tilde{b}, & t \in (0,1/2] \end{array} \right. .
\]
It is clear that $\mathrm{rank}(a_t(x)) \leq k-l$, whence $t \mapsto a_t \oplus q$ is a path in $S$ connecting $a$ and $b$, as desired.
\end{proof}

\begin{rems} {\rm  A continuous map $f: \mathrm{S}^k \to \mathrm{M}_n(\mathrm{C}(X))$ is naturally identified with an element of
$\mathrm{M}_n(\mathrm{C}(X \times \mathrm{S}^k))$.  It follows that if $k,l$ as in Proposition \ref{homandrank} satisfy $k-l \geq 4\mathrm{dim}(X) +4r$, then any two such maps are homotopic in $S$, so that $\pi_r(S)$ vanishes.   }
\end{rems}

\begin{lms}\label{localextend}
Let $X$ be a compact metric space, and let $Y \subseteq X$ be closed.  Let $f,g:X \to \mathbb{Z}^+$ be bounded functions which are lower semicontinuous and upper semicontinuous, respectively.  Assume that $f(x) \geq g(x)$ for each $x \in X$, and let $a \in \mathrm{M}_n(\mathrm{C}(Y))$ be positive and satisfy
\[
g(y) \leq \mathrm{rank}(a(y)) \leq f(y), \ \forall y \in Y.
\]
It follows that there are an open set $U \supseteq Y$ and a positive element $b \in \mathrm{M}_n(\mathrm{C}_b(U))$ such that $b|_Y=a$ and 
\[
g(z) \leq \mathrm{rank}(b(z)) \leq f(z), \ \forall z \in U.
\]
\end{lms}

\begin{proof}
By Tietze's Extension Theorem we can find an open set $V \supseteq Y$ and a positive element $\tilde{a} \in \mathrm{M}_n(\mathrm{C}_b(V))$ such that $\tilde{a}|_Y=a$.  The map $z \mapsto \mathrm{rank}(\tilde{a}(z))$ is lower semicontinuous on $V$, and so for each $y \in Y$ there is an open neighbourhood $W_y$ of $y$ in $V$ with the property that 
\[
\mathrm{rank}(\tilde{a}(z)) \geq \mathrm{rank}(\tilde{a}(y)), \ \forall z \in W_y.
\]
The function $g$, on the other hand, is upper semicontinuous, and so for each $y \in Y$ there is an open neighbourhood $U_y$ of $y$ in $V$ with the property that
\[
g(z) \leq g(y), \ \forall z \in U_y.
\]
Setting $E_y = W_y \cap U_y$ we have an open cover $\{E_y\}_{y \in Y}$ of $Y$ which has the property that 
\begin{equation}\label{lowerbd}
g(z) \leq g(y) \leq \mathrm{rank}(\tilde{a}(y)) \leq \mathrm{rank}(\tilde{a}(z)), \ \forall z \in \bigcup_{y \in Y} E_y.
\end{equation}
In other words, setting $U = \cup_{y \in Y} E_y$, we have an extension $\tilde{a}$ of $a$ to $U$ which satisfies the lower bound required by the conclusion of the Lemma.

Let $n_1 < n_2 < \cdots < n_k$ be the values taken by $f$.  Set 
\[
E_i = \{ x \in X \ | \ f(x) \leq n_i\}, \ 1 \leq i \leq k,
\]
and note that each $E_i$ is closed.  We set $E_0 = \emptyset$ as a notational convenience.  Let us take $\tilde{a}$ and $U$ as above;  by shrinking $U$ slightly, we may assume that $\tilde{a}$ is defined on $\overline{U}$.  Also, combining (\ref{lowerbd}) with Lemma \ref{ranketa}, we can find $\eta > 0$ such that 
\begin{equation}\label{ranklowerbd}
\mathrm{rank}((\tilde{a}-\eta^{'})_+(x)) \geq g(x), \ \forall x \in \overline{U}, 0 < \eta^{'} \leq \eta.
\end{equation}

The uniform continuity of $\tilde{a}$ on $\overline{U}$ implies that for each $n \in \mathbb{N}$ there is $\delta_n>0$ such that for each $x \in \overline{U}$ and $y \in Y$ satisfying $\mathrm{dist}(x,y)< \delta_n$ we have $\|\tilde{a}(x)-\tilde{a}(y)\|< \eta/2^n$.  Let $i_1 < i_2 < \cdots < i_t$ be the indices for which $Y \cap (E_{i_l}\backslash E_{i_l-1}) \neq \emptyset$.  Set $\delta_n^{(1)} = \delta_n$ and
\[
U_n^1 = \{x \in X \ | \ \mathrm{dist}(x,Y \cap (E_{i_1} \backslash E_{i_1-1})) <\delta_n^{(1)} \}.
\]
Suppose that we have found, for some $r < t$, open sets $U_n^1$ (as above), $U_n^2, \ldots U_n^r$ and positive tolerances $\delta_n^{(1)}, \ldots, \delta_n^{(r)} < \delta$ with the following properties:  
\begin{itemize}
\item[$\bullet$]
for each $s \leq r$,
\[
U_n^s = \{x \in X \ | \ \mathrm{dist}(x,(Y \cap (E_{i_s} \backslash E_{i_s-1})) \backslash (\cup_{l < s} U_n^l)) < \delta_n^{(s)}\}
\]
\item[$\bullet$] $U_n^s \cap E_{i_s-1} = \emptyset$.
\end{itemize}
Since $\cup_{l < r+1} U_n^l$ contains $Y \cap (E_{i_r} \backslash E_{i_r-1})$ and is open, we see that 
\[
(Y \cap (E_{i_{r+1}} \backslash E_{i_{r+1}-1})) \backslash (\cup_{l < r+1} U_n^l)
\]
is a closed subset of $E_{i_{r+1}} \backslash E_{i_{r+1}-1}$, and there is therefore $0< \delta_n^{(r+1)} < \delta_n$ such that the bullet points above hold with $s = r+1$, too.  Continuing in this manner we arrive at open sets $U_n^1,\ldots,U_n^t$, and we set $U_n = \cup_{l \leq t} U_n^l \supseteq Y$.  Let us fix $U_1$, and note that by shrinking the tolerances $\delta_n^{(l)}$ used to construct $U_n$ if necessary, we may assume that $\overline{U_{n+1}} \subseteq U_n$ for each $n \in \mathbb{N}$.  From our bullet points we extract the following fact:
\begin{enumerate}
\item[(i)] for each $1 < i \leq k$, for each $x \in \overline{U_{n}} \cap (E_i \backslash E_{i-1})$, there are $j \leq i$ and $y \in Y \cap (E_j \backslash E_{j-1})$ such that $\|\tilde{a}(x)-\tilde{a}(y)\| \leq \eta/2^n$.
\end{enumerate}
Fix a continuous function $f:\overline{U_1} \to [0,1]$ with the following properties:
\begin{enumerate}
\item[(ii)] $f(y)=0, \ \forall y \in Y$;
\item[(iii)] $\eta > f(x) \geq \eta/2^{n-1}, \ \forall x \in \overline{U_n} \backslash \overline{U_{n+1}}$.
\end{enumerate}
Now define $b(x) = (\tilde{a}(x)-f(x))_+$ for each $x \in \overline{U_1}$, and note that $a:\overline{U_1} \to \mathrm{M}_n$ is continuous since $f$ is.

If $y \in Y$ then $b(y) = \tilde{a}(y) = a(y)$ by (ii), and the desired rank inequality holds for $b$ by assumption.  If $x \in \overline{U_1} \backslash Y$, then (iii) and (\ref{ranklowerbd}) imply that $\mathrm{rank}(b(x)) \geq g(x)$.  It remains to establish our upper bound for such $x$.
If $x \in (\overline{U_n} \backslash Y) \cap (E_i \backslash E_{i-1})$, then by (i) there are $j \leq i$ and $y \in Y \cap (E_j \backslash E_{j-1})$ such that $\|\tilde{a}(x)-\tilde{a}(y)\|< \eta/2^n$.  Combining this with (iii) yields
\[
\mathrm{rank}(b(x)) = \mathrm{rank}((\tilde{a}(x) - f(x))_+) \leq  \mathrm{rank}(\tilde{a}(y)) = \mathrm{rank}(a(y)) = n_j  \leq n_i = f(x).
\]
Replacing $U$ with $U_1$ completes the proof.

\end{proof}


\begin{props}\label{globalextend}
Let $X$ be a compact metric space, and let $Y \subseteq X$ be closed.  Let $k,l \in \mathbb{N}$ satisfy $k-l \geq 4 \mathrm{dim}(X)$.  Suppose that $a \in \mathrm{M}_n(\mathrm{C}(Y))_+$ satisfies
\[
l \leq \mathrm{rank}(a(y)) \leq k, \ \forall y \in Y.
\]
It follows that there is a positive element $b \in \mathrm{M}_n(\mathrm{C}(X))$ such that $b|_Y=a$ and 
\[
l \leq \mathrm{rank}(b(x)) \leq k, \ \forall x \in X.
\]
\end{props}

\begin{proof}This is a more or less standard argument.  Using Lemma \ref{localextend} we may assume that $a$ is defined on the closure $\overline{U}$ of an open superset $U$ of $Y$, and that $a$ still satisfies the required rank bounds on $\overline{U}$.  Fix an open set $V$ in $X$ such that $Y \subseteq V \subseteq \overline{V} \subseteq U$ and a continuous map $f:X \to [0,1]$ such that $f|_{\overline{V}}=0$ and $f|_{U^c}=1$.  Fix a positive element $d$ of $\mathrm{M}_n(\mathrm{C}(V^c))$ such that 
\[
l \leq \mathrm{rank}(d(x)) \leq k, \ \forall x \in V^c.
\]
Apply Lemma \ref{homandrank} to find a path $h(t)$ between $a|_{V^c}$ and $d$ satisfying the requisite rank bounds.  Finally, define
$b(x) = h(f(x))$. 
\end{proof}

\begin{props}\label{envelopeextend}
Let $X$ be a compact metric space for which $\mathrm{dim}(X)<\infty$, and let $Y \subseteq X$ be closed.  
Let $f,g:X \to \mathbb{Z}^+$ 
be bounded functions which are lower semicontinuous and upper semicontinuous, respectively, and suppose that $f(x)-g(x) \geq 4\mathrm{dim}(X)$ for each $x \in X$.  Let $a \in \mathrm{M}_n(\mathrm{C}(Y))_+$ satisfy
\[
g(y) \leq \mathrm{rank}(a(y)) \leq f(y), \ \forall y \in Y.
\]
It follows that there is $b \in \mathrm{M}_n(\mathrm{C}(X))_+$ such that $b|_Y=a$ and
\begin{equation}\label{envbound}
g(x) \leq \mathrm{rank}(b(x)) \leq f(x), \ \forall x \in X.
\end{equation}
\end{props}

\begin{proof}
Let $n_1 < n_2 < \cdots < n_k$ be the values attained by $f$, and set $E_i = \{ x \in X \ | \ f(x) \leq n_i \}$.  The $E_i$s then are closed.  
It suffices to consider the case where $f$ is constant on $X \backslash Y$, for if this case of the Proposition holds, then we may apply it successively to extend $a$ from $Y$ to $Y \cup E_1$, from $Y \cup E_1$ to $Y \cup E_2$, and so on.

Let $m_1 > m_2 > \cdots > m_k$ be the values taken by $g$, and let $r$ be the value attained by $f$ on $X \backslash Y$.  Set
\[
F_j = \{ x \in X \ | \ g(x) \geq m_j \}.
\]
Note that each $F_j$ is closed, and that $F_j \supseteq F_{j+1}$.  We will construct $b$ by making successive extensions to $Y \cup F_1, Y \cup F_2, \ldots, Y \cup F_k = X$.  First consider $Y$ as a closed subset of $Y \cup F_1$.  Using Lemma \ref{localextend}, we can extend $a$ to an open subset $U$ of $Y \cup F_1$ containing $Y$ in such a manner that the extension satisfies (\ref{envbound}) for each $x \in U$;  by shrinking $U$ slightly, we can assume that $a$ is defined and satisfies the said bounds on $\overline{U}$.  Let $V$ be an open subset of $Y \cup F_1$ such that $Y \subseteq V$ and $\overline{V} \subseteq U$.  Now extend $a|_{V^{c} \cap \overline{U}}$ to all of $V^{c}$ using Proposition \ref{globalextend}, so that the extension satisfies (\ref{envbound}), too.  This completes the extension of $a$ to $Y \cup F_1$.  The remaining extensions are carried out in the same manner.

\end{proof}

\section{Proof of Theorem \ref{waz}}\label{main}

Let us recall some of the terminology and results from \cite{P}.
A C$^*$-algebra $R$ is a {\it recursive subhomogeneous (RSH) algebra} if it can be written as an iterated pullback of the
following form:
\begin{equation}\label{decomp2}
R = \left[ \cdots \left[  \left[ C_0 \oplus_{C_1^{(0)}} C_1 \right] \oplus_{C_2^{(0)}} C_2 \right]
\cdots \right] \oplus_{C_l^{(0)}} C_l,
\end{equation}
with $C_k = \mathrm{M}_{n(k)}(\mathrm{C}(X_k))$ for compact Hausdorff spaces $X_k$ and integers
$n(k)$, with $C_k^{(0)}=\mathrm{M}_{n(k)}(\mathrm{C}(X_k^{(0)}))$ for compact subsets $X_k^{(0)}
\subseteq X$ (possibly empty), and where the maps $C_k \to C_k^{(0)}$ are always the restriction maps.  
We refer to the expression in (\ref{decomp2}) as a {\it decomposition} for $R$.  Decompositions for RSH
algebras are not unique.

Associated with the decomposition (\ref{decomp2}) are:
\begin{enumerate}
\item[(i)] its {\it length} $l$;
\item[(ii)] its {\it $k^{\mathrm{th}}$ stage algebra}
\[
R_k = \left[ \cdots \left[  \left[ C_0 \oplus_{C_1^{(0)}} C_1 \right] \oplus_{C_2^{(0)}} C_2
\right] \cdots \right] \oplus_{C_k^{(0)}} C_k;
\]
\item[(iii)] its {\it base spaces} $X_0,X_1,\ldots,X_l$ and {\it total space} $X:=\sqcup_{k=0}^l X_k$;
\item[(iv)] its {\it matrix sizes} $n(0),n(1),\ldots,n(l)$ and {\it matrix size function $m:X \to
\mathbb{N}$} given by $m(x) = n(k)$ when $x \in X_k$ (this is called the {\it matrix size of $R$ at $x$});
\item[(v)] its {\it topological dimension $\mathrm{dim}(X)$} and {\it topological dimension function $d:X \to \mathbb{N} \cup
\{0\}$} given by $d(x) = \mathrm{dim}(X_k)$ when $x \in X_k$;
\item[(vi)] its {\it standard representation $\sigma_R:R \to \oplus_{k=0}^l \mathrm{M}_{n(k)}(\mathrm{C}(X_k))$} defined to be the obvious inclusion;
\item[(vii)] the {\it evaluation maps $\mathrm{ev}_x:R \to \mathrm{M}_{n(k)}$} for $x \in X_k$, defined to be the composition of evaluation
at $x$ on $\oplus_{k=0}^l \mathrm{M}_{n(k)}(\mathrm{C}(X_k))$ and $\sigma_R$.
\end{enumerate}

\begin{rems}\label{rshmisc} {\rm
  If $R$ is separable, then the $X_k$ can be taken to be
metrizable.  It is clear from the construction of $R_{k+1}$ as a pullback of $R_k$ and $C_{k+1}$ that there is a
canonical surjective $*$-homomorphism $\lambda_k:R_{k+1} \to R_k$.  By composing several such, one has also a canonical surjective 
$*$-homomorphism from $R_j$ to $R_k$ for any $j >k$.  Abusing notation slightly, we denote these maps by $\lambda_k$ as well.
The C$^*$-algebra $\mathrm{M_m}(R) \cong R \otimes \mathrm{M}_m(\mathbb{C})$ is an RSH algebra in a canonical way.}
\end{rems}


Each unital separable ASH algebra is the limit of an inductive system of RSH algebras by the main result of \cite{NW}, whence the following definition of slow dimension growth is sensible.

\begin{dfs}\label{sdg}
Let $(A_i,\phi_i)_{i \in \mathbb{N}}$ be a direct system of RSH algebras with each $\phi_i:A_i \to A_{i+1}$ a unital $*$-homomorphism.  Let $l_i$ be the length of $A_i$, $n_i(0),n_i(1),\ldots,n_i(l_i)$ its matrix sizes, and $X_{i,0},X_{i,1},\ldots,X_{i,l_i}$ its base spaces.  We say that the system $(A_i,\phi_i)$ has slow dimension growth if
\[
\limsup_{i} \left( \mathrm{max}_{0 \leq j \leq l_i}  \ \frac{\mathrm{dim}(X_{i,j})}{n_i(j)} \right)= 0.
\]
If $A$ is a unital ASH algebra, then we say it has slow dimension growth if it can be written as the limit of a slow dimension growth system as above.
\end{dfs}
\noindent
This definition is equivalent to that of Phillips (\cite[Definition 1.1]{P2}) for simple algebras, as shown by the proof of \cite[Theorem 5.3]{To4}.  It is, however, suitable only for simple algebras.

For an inductive system as above with limit algebra $A$, we let $\phi_{i \infty}:A_i \to A$ denote the canonical $*$-homomorphism.
Before proving the main result of this section, we need one more Lemma.

\begin{lms}\label{semicon}
Let $X$ be a topological space and let $\alpha:X \to \mathbb{R}$ be bounded and continuous.  Given $n \in \mathbb{N}$, define maps $\overline{\alpha}_n, \underline{\alpha}_n:X \to \mathbb{R}$ as follows:  
\begin{itemize}
\item[$\bullet$] $\underline{\alpha}_n$ is the largest lower semicontinuous function on $X$ which takes values in $\{k/n \ | \ k \in \mathbb{Z}\}$ and satisfies $\underline{\alpha}_n \leq \alpha$; 
\item[$\bullet$] $\overline{\alpha}_n$ is the smallest upper semicontinuous function on $X$ which takes values in $\{k/n \ | \ k \in \mathbb{Z} \}$ and satisfies $\overline{\alpha}_n \geq \alpha$.  
\end{itemize}

It follows that if $f:X \to \mathbb{R}$ is any function taking values in $\{k/n \ | \ k \in \mathbb{Z}\}$ and satisfying $f \geq \alpha$,then $f \geq \overline{(\alpha-\delta)}_n$ for each $\delta>0$;  and clearly if $f$ is instead lower semicontinuous and satisfies $f \leq \alpha$, then $f \leq \underline{\alpha}_n$.  Finally, we have $|f(x)-\underline{\alpha}_n(x)|, \ |f(x)-\overline{\alpha}_n| < 2/n$.
\end{lms}

\begin{proof}
Families of lower semicontinuous functions are closed under taking pointwise suprema.  Since $\alpha$ is bounded, the set of lower semicontinuous functions $f$ on $X$ taking values in $\{k/n \ | \ k \in \mathbb{Z} \}$ and satisfying $f \leq \alpha$ is not empty, and so $\underline{\alpha}_n$ exists.  A similar argument using infima of upper semicontinuous functions establishes the existence of $\overline{\alpha}_n$.  

Define $h_\alpha:X \to \{k/n \ | \ k \in \mathbb{Z} \}$ as follows:  if $k/n \leq \alpha(x) < (k+1)/n$, then $h(x) = (k+1)/n$.  It is straightforward to check that $h_\alpha$ is upper semicontinuous and $h_\alpha > \alpha(x)$ by definition.  If follows that $h_\alpha \geq \overline{\alpha}_n$.  Let $f$ be any function on $X$ taking values in $\{k/n \ | \ k \in \mathbb{Z}\}$ and satisfying $f \geq \alpha$.  
If $k/n < \alpha(x) < (k+1)/n$, then we must have 
\[
f(x) \geq (k+1)/n = h_\alpha(x) \geq \overline{\alpha_n}(x) \geq \overline{(\alpha-\delta)}_n(x).
\]
If $f(x) = k/n$, then $\alpha(x) \leq k/n$.  It follows that $(\alpha-\delta)(x) < k/n$, and so 
\[
f(x) = k/n \geq h_{\alpha-\delta}(x) \geq \overline{(\alpha-\delta)}_n(x).
\] 

By construction, we have
\[
0 \leq \overline{\alpha}_n(x)-\alpha(x)  = |\overline{\alpha}_n(x)-\alpha(x)| \leq h_\alpha(x) -\alpha(x) \leq 1/n < 2/n.
\]
For the other estimate, define $g_\alpha:X \to \{k/n \ | \ k \in \mathbb{Z} \}$ as follows:  if $k/n < \alpha(x) \leq (k+1)/n$, then set $g_\alpha(x) = k/n$.  It is straightforward to check that $g_\alpha$ is lower semicontinuous, and so
\[
0 \leq \alpha(x)-\underline{\alpha}_n(x) = |\alpha(x) -\underline{\alpha}_n(x)| \leq \alpha(x) - g_\alpha(x) \leq 1/n < 2/n.
\]

\end{proof}

\begin{thms}\label{ashrange}
Let $(A_i,\phi_i)$ be a direct system of RSH algebras with slow dimension growth, and 
let $A = \lim_i(A_i,\phi_i)$.  Assume that $A$ is simple.  Let $f$ be a strictly positive 
affine continuous function on $\mathrm{T}(A)$ and let $\epsilon>0$ be given.  It follows 
that there are $i_0,k \in \mathbb{N}$ and a positive element $a \in \mathrm{M}_k(A_{i_0})$ 
with the property that
\[
| f(\tau) - d_{\tau}(\phi_{i_0 \infty}(a)) | < \epsilon.
\]
\end{thms}

\begin{proof}
For a compact metrizable Choquet simplex $K$ we let $\mathrm{Aff}(K)$ denote the set of 
continuous affine functions on $K$.  Each $\phi_i:A_i \to A_{i+1}$ induces a continuous 
affine map $\phi_i^{\sharp}:\mathrm{T}(A_{i+1}) \to \mathrm{T}(A_i)$ and a dual map
\[
\phi_i^{\bullet}:\mathrm{Aff}(\mathrm{T}(A_i)) \to \mathrm{Aff}(\mathrm{T}(A_{i+1}))
\]
given by $\phi_i^{\bullet}(f)(\tau) = f(\phi_i^{\sharp}(\tau))$.  It is well known that 
$\cup_{i \in \mathbb{N}} \ \phi_{i \infty}^{\bullet}(\mathrm{Aff}(\mathrm{T}(A_i)))$ is uniformly 
dense in $\mathrm{Aff}(\mathrm{T}(A))$, so we may assume that $f = \phi_{i \infty}^{\bullet}(g)$ 
for some $i \in \mathbb{N}$ and $g \in \mathrm{Aff}(\mathrm{T}(A_i))$.  Truncating and re-labeling 
our inductive sequence, we may assume that $i=1$.  We may also assume, by replacing $A$ with 
$\mathrm{M}_k(A)$ for some large enough $k \in \mathbb{N}$, that $\|g\| \leq 1$.  We shall also 
assume that $\epsilon <\|g\|$.

Set $\phi_{i,j} = \phi_{j-1} \circ \cdots \circ \phi_i$, and assume, contrary to our desire, that for each $j \geq 1$, for some $\tau_j \in \mathrm{T}(A_j)$, we have $\phi_{1,j}^{\bullet}(g)(\tau_j) \leq 0$.  Let 
\[
\gamma_j = (\phi_{1,j}^{\sharp}(\tau_j), \phi_{2,j}^{\sharp}(\tau_j),\ldots, \phi_{j-1,j}^{\sharp}(\tau_j), \tau_j, \tau_{j+1}, \tau_{j+2}, \ldots) \in \prod_{i = 1}^{\infty} \mathrm{T}(A_i).
\]
Since $\prod_{i=1}^{\infty} \mathrm{T}(A_i)$ is compact, the sequence $(\gamma_j)$ has a subsequence converging to some
$\eta = (\eta_1,\eta_2,\eta_3, \ldots)$.  Let $\gamma_j^{(i)}$ denote the $i^{\mathrm{th}}$ entry of the sequence $\gamma_j$.  We have that $\phi_{i-1,i}^{\sharp}(\gamma_j^{(i)}) = \gamma_j^{(i-1)}$ for each $i \leq j$, whence $\phi_{i-1,i}^{\sharp}(\eta_i) = \eta_{i-1}$ for each $i \in \mathbb{N}$.  It follows that $\eta$ defines an element of $\mathrm{T}(A)$, and that 
\[
f(\eta) = \phi_{1 \infty}^{\bullet}(g)(\eta) = g(\phi_{1 \infty}^{\sharp}(\eta)) = g(\eta_1).
\]
Now $\eta_1$ is a subsequential limit of the sequence $(\phi_{1,j}^{\sharp}(\tau_j))_{j \in \mathbb{N}}$, and we have
\[
g(\phi_{1,j}^{\sharp}(\tau_j)) = \phi_{1,j}^{\bullet}(g)(\tau_j) \leq 0, \ \forall j \in \mathbb{N}.
\]
It follows that $0 \geq g(\eta_1) = f(\eta)$, contradicting the strict positivity of $f$.  We conclude that for some $j_0 \in \mathbb{N}$, for each $\tau \in \mathrm{T}(A_{j_0})$, we have $\phi_{1, j_0}^{\bullet}(g)(\tau) > 0$.  Let us once again truncate and re-label our sequence, so that $j_0=1$.  

To complete the proof of the theorem, it will suffice to find $i_0 \geq 1$ and a positive element $b \in A_{i_0}$ with the property that 
\[
| \phi_{1, i_0}^{\bullet}(g)(\tau) - d_\tau(b) | < \epsilon, \ \forall \tau \in \mathrm{T}(A_{i_0});
\]
it is then straightforward to check that $a = \phi_{i_0 \infty}(b)$ has the required property.  Let us set $g_i = \phi_{1,i}^{\bullet}(g)$ for convenience.  Using the slow dimension growth of the system $(A_i,\phi_i)$ and the simplicity of the limit, find $i_0$ large enough that 
\[
(\epsilon/4) n_{i_0}(j) > 4 \mathrm{dim}(X_{i_0,j})+4, \ 0 \leq j \leq l_{i_0}.
\]
(The simplicity of $A$ guarantees that all of the $n_{i_0}(j)$s are large for large enough $i_0$, and this makes the "$+4$" term on the right hand side possible.)  From here on we will work only inside $A_{i_0}$, so let us avoid subscripts by re-naming this algebra B, re-naming its matrix sizes $n(0),\ldots,n(l)$, re-naming its base spaces $X_0,\ldots,X_l$, and setting $h = \phi_{1,i_0}^{\bullet}(g)$.  Thus, we have 
\begin{equation}\label{dimbound}
(\epsilon/4) n(j) > 4 \mathrm{dim}(X_{j})+4, \ 0 \leq j \leq l.
\end{equation}

The function $h$ defines strictly positive continuous functions $h_0,h_1,\ldots,h_l$ on $X_0,X_1,\ldots,X_l$, respectively, via its standard representation $\sigma$---the function $h_j$ is the restriction of $\sigma^{\bullet}(h)$ to $X_j$.  We need only find a positive element $b \in B$ such that 
\begin{equation}\label{toprove}
| h_j(x)-d_{\tau_x}(b) | < \epsilon, \ \forall x \in X_j \backslash X_j^{(0)},
\end{equation}
where $\tau_x$ denotes the extreme tracial state corresponding to evaluation at $x \in X_j \backslash X_j^{(0)}$ composed with the usual trace on $\mathrm{M}_{n(j)}$.  Combining Lemma \ref{semicon} with (\ref{dimbound}) we have
\[
\left|h_j(x) - \underline{h_j}_{n(j)}(x)\right|, \ \left|\overline{(h_j-3\epsilon/4)}_{n(j)}(x) - (h_j - 3\epsilon/4)(x)\right| < 2/n(j) < \epsilon/4, \ \forall x \in X_j,
\]
from which we extract, for $\delta \geq 3\epsilon/4$,
\begin{eqnarray}\label{dimbound2}
n(j) \left( \underline{h_j}_{n(j)}(x) - \overline{(h_j-\delta)}_{n(j)}(x) \right) & \geq& n(j)\left(h)j(x) - (h_j-\delta)(x) -4/n(j)\right) \\
& > & n(j)\delta - 4 \\
& \geq & n(j)3\epsilon/4 - 4 \\
& \stackrel{(\ref{dimbound})}{>}  & 4\mathrm{dim}(X_j)
\end{eqnarray}
for each $x \in X_j$.
We also have
\[
\left| \underline{h_j}_{n(j)}(x) - \overline{(h_j-\delta)}_{n(j)}(x) \right| \leq |h_j(x) - (h_j-\delta)(x)| \leq \delta, \ \forall x \in X_j.
\]

Fix strictly positive tolerances $\delta_0,\delta_1,\ldots,\delta_l$ such that $3\epsilon/4 \leq \delta_0$ and $\sum_{j=0}^l \delta_j < \epsilon$. Set $\eta_k = \sum_{j=0}^k \delta_j$.  Suppose that we have found a positive element $b_j$ of the $j^{\mathrm{th}}$ stage algebra $B_j$ (see terminology at the beginning of this Section) with the following property:
\begin{equation}\label{rankineq3}
(h_k-\eta_j)(x) \leq \mathrm{rank}(b_j(x))/n(k) \leq h_k(x), \ \forall x \in X_k, \    0 \leq k \leq j.
\end{equation}
It follows immediately that this same rank inequality holds with $k = j+1$ at each $x \in X_{j+1}^{(0)}$ by pushing forward with the map $\psi_j^{\bullet}$ induced by the clutching map $\psi_j$ (let us assume that the definition of $b_j$ over $X_{j+1}^{(0)}$ is given by pushing $b_j$ forward via $\psi_j$).  An application of Lemma \ref{semicon} and the fact that $\eta_{j+1} > \eta_j$ then gives
\[
\overline{(h_{j+1}-\eta_{j+1})}_{n(j+1)}(x) \leq \mathrm{rank}(b_j(x))/n(j+1) \leq \underline{h_{j+1}}_{n(j+1)}(x), \ \forall x \in X_{j+1}^{(0)}.
\]
We can now use (\ref{dimbound2}) and the inequality above in order to apply Proposition \ref{envelopeextend} with $X = X_{j+1}$, $Y = X_{j+1}^{(0)}$, and $a = b_j$ to find a positive element $b_{j+1} \in \mathrm{M}_{n(j+1)}(\mathrm{C}(X_{j+1}))$ which restricts to $b_j$ on $X_{j+1}^{(0)}$ and satisfies
\[
\overline{(h_{j+1}-\eta_{j+1})}_{n(j+1)}(x) \leq \mathrm{rank}(b_{j+1}(x))/n(j+1) \leq \underline{h_{j+1}}_{n(j+1)}(x), \ \forall x \in X_{j+1}.
\]
Abusing notation slightly, we let $b_{j+1}$ denote the positive element of $B_{j+1}$ which restricts to $b_j$ in $B_j$ and which agrees with the element $b_{j+1}$ above over $X_{j+1}$.  Now (\ref{rankineq3}) holds with $j+1$ in place of $j$.  Note that the existence of an appropriate $b_0$ follows from an application of Proposition \ref{envelopeextend} with $Y = \emptyset$, so that iteration of the process we have just described will lead to a positive element $b$ of $B$ satisfying
\[
(h_k-\eta_l)(x) \leq \mathrm{rank}(b(x))/n(k) \leq h_k(x), \ \forall x \in X_k, \    0 \leq k \leq l.
\]
If $x \in X_k \backslash X_k^{(0)}$, then $d_{\tau_x}(b) = \mathrm{rank}(b(x))/n(k)$.  This gives (\ref{toprove}), completing the proof of the Theorem.

\end{proof}

 Now we can prove our main results.  Throughout, $A$ is a unital simple separable ASH algebra with slow dimension growth.
 
 \vspace{2mm}
 \noindent
 \begin{proof} {\it (Theorem \ref{waz})} By \cite[Theorem 1.1]{To4}, $A$ has strict comparison of positive elements.  Let $\mathrm{SAff}(\mathrm{T}(A))$ denote the set of suprema of increasing sequences of continuous affine strictly positive functions on $\mathrm{T}(A)$.  By \cite[Theorem 2.5]{BT} and the comment thereafter, we know that
 \begin{eqnarray*}
 W(A \otimes \mathcal{Z}) & \cong & V(A \otimes \mathcal{Z}) \sqcup \mathrm{SAff}(\mathrm{T}(A \otimes \mathcal{Z})) \\
 & \cong & \mathrm{K}_0(A \otimes \mathcal{Z})_+ \sqcup \mathrm{SAff}(\mathrm{T}(A \otimes \mathcal{Z})) 
 \end{eqnarray*}
 where $V(A)$ denotes the Murray-von Neumann semigroup.  On the other hand, \cite[Theorem 2.5]{BT} (or rather, its proof) shows that any unital simple exact tracial C$^*$-algebra $B$ which has strict comparison of positive elements and the property that any strictly positive $f \in \mathrm{Aff}(\mathrm{T}(B))$ is uniformly arbitrarily close to a function of the form
 $\tau \mapsto d_\tau(a)$ for some $a \in (B \otimes \mathcal{K})_+$ must then also satisfy
  \begin{equation}\label{wform}
 W(B) \cong V(B) \sqcup \mathrm{SAff}(\mathrm{T}(B)).
 \end{equation}
 Using Theorem \ref{ashrange}, we see that (\ref{wform}) holds with $B=A$.  
 By \cite[Theorem 0.1 (2)]{P2}, $A$ has cancellation of projections, whence $V(A) \cong  \mathrm{K}_0(A)_+$.  Since $A$ is simple and has strict comparison, its $\mathrm{K}_0$-group is unperforated, so $\mathrm{K}_0(A)_+ \cong \mathrm{K}_0(A \otimes \mathcal{Z})_+$ by a result of Gong, Jiang, and Su (\cite{GJS}).  It is also known that $\mathrm{T}(A) \cong \mathrm{T}(A \otimes \mathcal{Z})$, since $\mathcal{Z}$ admits a unique normalised tracial state.  To conclude that $W(A) \cong W(A \otimes \mathcal{Z})$ we need to prove that the now obvious identification of these semigroups as sets is in fact an isomorphism of ordered semigroups.  This, however, follows from the description of the order and addition (for both cases) given in the comments preceding \cite[Theorem 2.5]{BT}.) 
  \end{proof}

\vspace{2mm}
\noindent
\begin{proof} {\it (Corollary \ref{zstabsdg})}
The forward implication follows from Theorems \ref{waz} and Winter's Theorem \ref{locfin}.  The reverse implication is Theorem 5.5 of \cite{TW}.
\end{proof}

\vspace{2mm}
\noindent
\begin{proof} {\it (Corollary \ref{class})} Let $\mathcal{C}$ denote the class of unital simple ASH algebras with slow dimension growth in which projections separate traces.  By Corollary \ref{zstabsdg}, each $A \in \mathcal{C}$ is $\mathcal{Z}$-stable, and so by the main result of \cite{LN} (based on \cite{Wi5}), we need only establish the conclusion of Corollary \ref{class} for the collection $\mathcal{C}^{'}$ consisting of algebras of the form $A \otimes \mathfrak{U}$ with $A \in \mathcal{C}$ and $\mathfrak{U}$ a UHF algebra of infinite type.  These algebras are $\mathcal{Z}$-stable unital simple ASH algebras with real rank zero, and so the desired classification result is given by Corollary 2.5 of \cite{Wi2}.
\end{proof}


\begin{thebibliography}{999}

\bibitem{BDR} Blackadar, B., Dadarlat, M., and R{\o}rdam, M.: {\it The
real rank of inductive limit $C^{*}$-algebras}, Math. Scand. {\bf 69} 
(1991), 211-216

\bibitem{BE} Bratteli, O., and Elliott, G. A.: {\it Small eigenvalue variation and real rank zero}, Pac. J. Math. {\bf 175} (1996), 47-59

\bibitem{BPT} Brown, N. P., Perera, F., and Toms, A. S.: {\it The Cuntz semigroup, the Elliott conjecture, and dimension functions on C$^*$-algebras}, J. reine angew. Math. {\bf 621} (2008), 191-211

\bibitem{BT} Brown, N. P., and Toms, A. S.: {\it Three applications of the Cuntz semigroup}, Int. Math. Res. Not., Vol. 2007, article ID rnm068, 14 pages

\bibitem{DNNP} Dadarlat, M., Nagy, G., Nemethi, A., and Pasnicu, C.: {\it Reduction of topological stable rank in inductive 
limits of $C^*$-algebras},
Pac. J. Math. {\bf 153} (1992), 267-276

\bibitem{D} Dadarlat, M.: {\it Reduction to dimension three of of local spectra of real rank zero $C^{*}$-algebras}, 
J.\ Reine Angew.\ Math.\ {\bf 460} (1995), 189-212

 \bibitem{DG} Dadarlat, M.\ and Gong., G.: {\it A classification result for
approximately homogeneous $C^*$-algebras of real rank zero}, Geom.\ Funct.\
Anal.\ {\bf 7} (1997), 646-711

\bibitem{E} Eilers, S.: {\it A complete invariant for AD algebras with real rank zero and bounded 
torsion in $\mathrm{K}_1$},  J. Funct. Anal. {\bf 139} (1996), 325-348

\bibitem{EG} Elliott, G. A., and Gong, G.: {\it On the classification of $C^*$-algebras of real rank zero. II.} Ann. of Math. (2) {\bf 144} (1996), 497-610

\bibitem{ET} Elliott, G. A., and Toms, A. S.: {\it Regularity properties in the classification program for separable amenable C$^*$-algebras}, Bull.
Amer. Math. Soc. {\bf 45} (2008), 229-245

\bibitem{G1} Gong, G.: {\it On inductive limits of matrix algebras over higher dimensional spaces, I}, Math. Scand. {\bf 80} (1997), 41-55

\bibitem{G2} Gong, G.: {\it On inductive limits of matrix algebras over higher dimensional spaces, II}, Math. Scand. {\bf 80} (1997), 56-100

\bibitem{GJS} Gong, G., Jiang, X., and Su, H.: {\it Obstructions to $\mathcal{Z}$-stability for unital simple C$^*$-algebras}, Can. Math. Bull. {\bf 43} (2000), 418-426

\bibitem{JS} Jiang, X.\ and Su, H.: {\it On a simple unital projectionless 
$C^{*}$-algebra}, Amer.\ J.\ Math.\ {\bf 121} (1999), 359-413

\bibitem{L} Lin, H.: {\it Classification of simple C$^*$-algebras of tracial topological rank zero}, Duke Math. J. {\bf 125} (2004), 99-137 

\bibitem{LN} Lin, H., and Niu, Z.: {\it Lifting KK-elements, asymptotical unitary equivalence, and classification of simple C$^*$-algebras}, Adv. Math. {\bf 219} (2008), 1729-1769

\bibitem{LP} Lin, Q., and Phillips, N. C.: {\it Ordered K-theory for C$^*$-algebras of minimal homeomorphisms}, 
in: Operator Algebras and Operator Theory, L. Ge, etc. (eds.), Contemp. Math. 228 (1998), 
289-314 


\bibitem{NW} Ng, P. W., and Winter, W.: {\it A note on Subhomogeneous $C^*$-algebras}, C. R. Math. Acad. Sci. Canada {\bf 28} (2006), 91-96 

\bibitem{P} Phillips, N. C.: {\it Recursive subhomogeneous algebras}, Trans. Amer. Math. Soc. {\bf 359} (2007), 4595-4623

\bibitem{P2} Phillips, N. C.: {\it Cancellation and stable rank for recursive subhomogeneous C$^*$-algebras}, Trans. Amer. Math. Soc. {\bf 359} (2007), 4625-4652

\bibitem{P3} Phillips, N. C.: {\it Every simple higher dimensional noncommutative torus is an A$\mathbb{T}$-algebra}, arXiv:math/0609783 (2006), preprint

\bibitem{R} R\o rdam, M.: {\it Classification of Cuntz-Krieger algebras}, K-theory {\bf 9} (1995), 31-58

\bibitem{T} Thomsen, K.: {\it The homoclinic and heteroclinic C$^*$-algebras of a generalized one-dimensional solenoid} arXiv:0809.1995 (2008), preprint


\bibitem{To1} Toms, A.\ S.: {\it On the independence of $\mathrm{K}$-theory
and stable rank for simple $C^*$-algebras}, J.\ Reine Angew.\ Math. {\bf 578} (2005), 185-199


 \bibitem{To2} Toms, A.\ S.: {\it On the classification problem for nuclear $C^*$-algebras}, Ann. of Math. (2) {\bf 167} (2008), 1059-1074
 
 \bibitem{To3} Toms, A.\ S.: {\it Stability in the Cuntz semigroup of a commutative C$^*$-algebra}, Proc. London Math. Soc. {\bf 96} (2008), 1-25
 
 \bibitem{To4} Toms, A.\ S.: {\it Comparison theory and smooth minimal C$^*$-dynamics}, Comm. Math. Phys. {\bf 289} (2009), 401-433
 
 \bibitem{TW} Toms, A.\ S., and Winter, W.: {\it Strongly self-absorbing C$^*$-algebras}, Trans. Amer. Math. Soc. {\bf 359} (2007), 3999-4029
 
 \bibitem{TW2} Toms, A.\ S., and Winter, W.: {\it Minimal dynamics and the classification of C$^*$-algebras}, Proc. Natl. Acad. Sci. USA {\bf 106} (2009), 16942-16943
 
 \bibitem{TW3} Toms, A.\ S., and Winter, W.: {\it Minimal dynamics and K-theoretic rigidity: Elliott's conjecture}, arXiv:0903:4133 (2009), preprint
 
 \bibitem{Wi1} Winter, W.: {\it Decomposition rank of subhomogeneous C$^*$-algebras}, Proc. London Math. Soc. {\bf 89} (2004), 427-456
 
 \bibitem{Wi2} Winter, W.: {\it Simple C$^*$-algebras with locally finite decomposition rank}, J. Funct. Anal. {\bf 243} (2006), 394-425
 
 
\bibitem{Wi3} Winter, W.: {\it Nuclear dimension and $\mathcal{Z}$-stability of perfect C$^*$-algebras}, arXiv:1006.2731 (2010)
 
 
 \bibitem{Wi4} Winter, W.: {\it Decomposition rank and $\mathcal{Z}$-stability}, Invent. Math. {\bf 179} (2010), 229-301
 
 \bibitem{Wi5} Winter, W.: {\it Localizing the Elliott conjecture at strongly self-absorbing C$^*$-algebras},
arXiv preprint \\ math.OA/0708.0283 (2007)
 
 
 
 \end{thebibliography}
 \end{document}